\documentclass{amsart}
\usepackage{amsmath,amscd,amsthm,amsfonts,amssymb}

\theoremstyle{plain}
\newtheorem{theorem}                 {Theorem}      [section]
\newtheorem{proposition}  [theorem]  {Proposition}

\newtheorem{lemma}        [theorem]  {Lemma}
\newtheorem{conjecture}        [theorem]  {Conjecture}

\theoremstyle{definition}
\newtheorem{example}      [theorem]  {Example}

\newtheorem{definition}   [theorem]  {Definition}

\numberwithin{equation}{section}

\def \theo-intro#1#2 {\vskip .25cm\noindent{\bf Theorem #1\ }{\it #2}}

\newcommand{\trace}{\operatorname{trace}}

\def \zn{\mathbb Z}

\def \rn{\mathbb R}
\def \cn{\mathbb C}

\def \B{\mathcal B}

\def \F{\mathcal F}
\def \H{\mathcal H}

\def \L{\mathcal L}

\def \V{\mathcal V}

\def \ip #1#2{\langle #1,#2 \rangle}

\def \a{\mathfrak{a}}

\def \d{\mathfrak{d}}
\def \g{\mathfrak{g}}
\def \h{\mathfrak{h}}
\def \k{\mathfrak{k}}
\def \m{\mathfrak{m}}
\def \n{\mathfrak{n}}
\def \p{\mathfrak{p}}

\def \s{\mathfrak{s}}

\def \v{\mathfrak{v}}
\def \z{\mathfrak{z}}

\DeclareMathOperator{\ad}{ad}
\DeclareMathOperator{\Ad}{Ad}

\def \GLR#1{\text{\bf GL}_{#1}(\rn)}

\def \glr#1{\mathfrak{gl}_{#1}(\rn)}

\def \SO#1{\text{\bf SO}(#1)}

\def \SOO#1#2{\text{\bf SO}(#1,#2)}

\def \Spin#1{\text{\bf Spin}(#1)}

\def \SU#1{\text{\bf SU}(#1)}

\def \Sp#1{\text{\bf Sp}(#1)}

\def \Spp#1#2{\text{\bf Sp}(#1,#2)}

\def \nab#1#2{\hbox{$\nabla$\kern -.3em\lower 1.0 ex
    \hbox{$#1$}\kern -.1 em {$#2$}}}

\begin{document}
\baselineskip 22pt \larger

\allowdisplaybreaks

\title{Harmonic morphisms from solvable Lie groups}

\author{Sigmundur Gudmundsson}
\author{Martin Svensson}



\address
{Department of Mathematics, Faculty of Science, Lund University,
Box 118, S-221 00 Lund, Sweden}
\email{Sigmundur.Gudmundsson@math.lu.se}

\address
{Department of Mathematics \& Computer Science, University of
Southern Denmark, Campusvej 55, DK-5230 Odense M, Denmark}
\email{svensson@imada.sdu.dk}

\dedicatory{Dedicated to the memory of Professor James Eells}

\begin{abstract}  In this paper we introduce two new methods for 
constructing harmonic morphisms from solvable Lie groups.  The first 
method yields global solutions from any simply connected nilpotent 
Lie group and from any Riemannian symmetric space of non-compact type 
and rank $r\ge 3$.  The second method provides us with global solutions
from any Damek-Ricci space and many non-compact Riemannian symmetric spaces.
We then give a continuous family of $3$-dimensional solvable Lie groups
not admitting any complex valued harmonic morphisms, not even locally.
\end{abstract}

\maketitle

\section{Introduction}

The notion of a minimal submanifold of a given ambient space is of
great importance in Riemannian geometry. Harmonic morphisms
$\phi:(M,g)\to(N,h)$ between Riemannian manifolds are useful tools
for the construction of such objects. They are solutions to
over-determined non-linear systems of partial differential
equations determined by the geometric data of the manifolds
involved. For this reason harmonic morphisms are difficult to find
and have no general existence theory, not even locally.

If the codomain is a surface the problem is invariant under
conformal changes of the metric on
$N^2$. Therefore, at least for local studies, the codomain can be
taken to be the complex plane with its standard flat metric. For
the general theory of harmonic morphisms between Riemannian
manifolds we refer to the self-contained book \cite{Bai-Woo-book}
and the regularly updated on-line bibliography \cite{Gud-bib}.

In \cite{Bai-Woo-1} Baird and Wood classify harmonic morphisms
from the famous $3$-dimensional homogeneous geometries to
surfaces. They construct a globally defined solution
$\phi:Nil\to\cn$ from the $3$-dimensional nilpotent Heisenberg
group and prove that the $3$-dimensional solvable group $Sol$ does
not admit any solutions, not even locally.

In this paper we introduce two new methods for constructing complex 
valued harmonic morphisms from solvable Lie groups.  We prove that 
any simply connected nilpotent Lie group can be equipped with 
left-invariant Riemannian metrics admitting globally
defined solutions.

Our methods yield a variety of harmonic morphisms from a large 
collection of solvable Lie groups for example all Damek-Ricci 
spaces.  We prove existence for most of the solvable 
Iwasawa groups $NA$ associated with Riemannian symmetric spaces.
Combining this with earlier results from \cite{Gud-Sve-1}, \cite{Gud-Sve-2} 
and \cite{Gud-Sve-3} we are able to prove the following conjecture except 
in the cases when the symmetric space is $G_2/\SO 4$ or its 
non-compact dual.

\begin{conjecture}\label{conj:existence}
Let $(M^m,g)$ be an irreducible Riemannian symmetric space of dimension 
$m\ge 2$.   For each point $p\in M$ there exists
a complex-valued harmonic morphism $\phi:U\to\cn$ defined on an open
neighbourhood $U$ of $p$. If the space $(M,g)$ is of non-compact
type then the domain $U$ can be chosen to be the whole of $M$.
\end{conjecture}

In the last section of the paper, we study the local existence of
harmonic morphisms on 3-dimensional solvable Lie groups and prove
the following result. 
\theo-intro{13.6}{Let $\g$ be a 3-dimensional centerless, solvable Lie
algebra and $G$ a connected Lie group with Lie algebra $\g$ and a
left-invariant metric. Let $\V$ be a local conformal foliation 
by geodesics on $G$. Then $G$ has constant sectional curvature.}

In contrast to the nilpotent case, we find a continuous family of 
solvable Lie groups not admitting any complex valued harmonic morphisms, 
independent of their left-invariant metric.  This family contains the 
famous example $Sol$. 

\section{Harmonic Morphisms}

Let $M$ and $N$ be two manifolds of dimensions $m$ and $n$,
respectively. A Riemannian metric $g$ on $M$ gives rise to the
notion of a {\it Laplacian} on $(M,g)$ and real-valued {\it
harmonic functions} $f:(M,g)\to\rn$. This can be generalized to
the concept of {\it harmonic maps} $\phi:(M,g)\to (N,h)$ between
Riemannian manifolds, which are solutions to a semi-linear system
of partial differential equations, see \cite{Bai-Woo-book}.

\begin{definition}
  A map $\phi:(M,g)\to (N,h)$ between Riemannian manifolds is
  called a {\it harmonic morphism} if, for any harmonic function
  $f:U\to\rn$ defined on an open subset $U$ of $N$ with $\phi^{-1}(U)$
non-empty,
  $f\circ\phi:\phi^{-1}(U)\to\rn$ is a harmonic function.
\end{definition}

The following characterization of harmonic morphisms between
Riemannian manifolds is due to Fuglede and Ishihara.  For the
definition of horizontal (weak) conformality we refer to
\cite{Bai-Woo-book}.

\begin{theorem}\cite{Fug-1,Ish}
  A map $\phi:(M,g)\to (N,h)$ between Riemannian manifolds is a
  harmonic morphism if and only if it is a horizontally (weakly)
  conformal harmonic map.
\end{theorem}

The following result of Baird and Eells gives the theory of
harmonic morphisms a strong geometric flavour and shows that the
case when $n=2$ is particularly interesting. The conditions
characterizing harmonic morphisms are then independent of
conformal changes of the metric on the surface $N^2$.

\begin{theorem}\cite{Bai-Eel}\label{theo:B-E}
Let $\phi:(M^m,g)\to (N^n,h)$ be a horizontally (weakly) conformal
map between Riemannian manifolds. If
\begin{enumerate}
\item[i.] $n=2$, then $\phi$ is harmonic if and only if $\phi$ has
minimal fibres at regular points; \item[ii.] $n\ge 3$, then two of the following
conditions imply the other:
\begin{enumerate}
\item $\phi$ is a harmonic map, \item $\phi$ has minimal fibres at regular points,
\item $\phi$ is horizontally homothetic.
\end{enumerate}
\end{enumerate}
\end{theorem}

In this paper we are interested in complex valued functions
$\phi,\psi:(M,g)\to\cn$ from Riemannian manifolds. In this
situation the metric $g$ induces the complex-valued Laplacian
$\tau(\phi)$ and the gradient $\text{grad}(\phi)$ with values in
the complexified tangent bundle $T^{\cn}M$ of $M$.  We extend the
metric $g$ to be complex bilinear on $T^{\cn} M$ and define the
symmetric bilinear operator $\kappa$ by
$$\kappa(\phi,\psi)= g(\text{grad}(\phi),\text{grad}(\psi)).$$ Two
maps $\phi,\psi: M\to\cn$ are said to be {\it orthogonal} if
$$\kappa(\phi,\psi)=0.$$  The harmonicity and horizontal
conformality of $\phi:(M,g)\to\cn$ are expressed by the relations
$$\tau(\phi)=0\ \ \text{and}\ \ \kappa(\phi,\phi)=0.$$

\begin{definition}  Let $(M,g)$ be a Riemannian
manifold.  Then a set $$\Omega=\{\phi_k:M\to\cn\ |\ k\in I\}$$ of
complex-valued functions is said to be an {\it orthogonal harmonic
family} on $M$ if, for all $\phi,\psi\in\Omega$,
$$\tau(\phi)=0\ \ \text{and}\ \ \kappa(\phi,\psi)=0.$$
\end{definition}

The problem of finding an orthogonal harmonic family on a Riemannian
manifold can often be reduced to finding a harmonic morphism with values
in $\rn^n$.

\begin{proposition} Let $\Phi:(M,g)\to \rn^n$ be a harmonic morphism
from a Riemannian manifold to the standard Euclidean
$\rn^n$ with $n\ge 2$.  If $V$ is an isotropic subspace of $\cn^n$ then
$$\Omega_V=\{\phi_v(x)=\ip {\Phi(x)}v|\ v\in V\}$$ is an orthogonal harmonic
family of complex valued functions on $(M,g)$.
\end{proposition}

Here and elsewhere in this paper $\ip\cdot\cdot$ refers to the standard
symmetric bilinear form on the $n$-dimensional complex linear space $\cn^n$.
The next result shows that orthogonal harmonic families can be useful
for producing a variety of harmonic morphisms.

\begin{theorem}\cite{Gud-Sve-1}\label{theo:local-sol}
Let $(M,g)$ be a Riemannian manifold and
$$\Omega=\{\phi_k:M\to\cn\ |\ k=1,\dots ,n\}$$ be a finite orthogonal
harmonic family on $(M,g)$.  Let $\Phi:M\to\cn^n$ be the map given
by $\Phi=(\phi_1,\dots,\phi_n)$ and $U$ be an open subset of
$\cn^n$ containing the image $\Phi(M)$ of $\Phi$.  If $\tilde\F$
is a family of holomorphic functions $F:U\to\cn$ then the family
$\F$ given by
$$\F=\{\psi:M\to\cn\ |\ \psi=F(\phi_1,\dots ,\phi_n),\ F\in\tilde\F\}$$
is an orthogonal harmonic family on $M$.
\end{theorem}

\section{The General Linear Group $\GLR n$}

Let $\GLR n$ be the general linear group equipped with its
standard Riemannian metric induced by the Euclidean scalar product
on the Lie algebra $\glr n$ given by
$$(X,Y)=\trace XY^t.$$  For $1\le i,j,k,l\le n$ we shall by
$E_{ij}$ denote the elements of $\glr n$ satisfying
$$(E_{ij})_{kl}=\delta_{ik}\delta_{jl}
\ \ \text{and}\ \ E_{ij}E_{kl}=\delta_{jk}E_{il}.$$

Let $G$ be a subgroup of $\GLR n$ with Lie algebra $\g$ equipped
with the induced Riemannian metric $g$.  If $X\in\g$ is a
left-invariant vector field on $G$ and $\phi,\psi:U\to\cn$ are two
complex valued functions defined locally on $G$ then
$$X(\phi)(p)=\frac {d}{ds}[\phi(p\cdot\exp(sX))]\big|_{s=0},$$
$$X^2(\phi)(p)=\frac {d^2}{ds^2}[\phi(p\cdot\exp(sX))]\big|_{s=0}.$$
This means that the operator $\kappa$ is given by
$$\kappa(\phi,\psi)=\sum_{X\in\B}X(\phi)X(\psi),$$ where $\B$ is
any orthonormal basis for the Lie algebra $\g$. Employing  the
Koszul formula for the Levi-Civita $\nabla$ connection on $G$ we
see that
\begin{eqnarray*}
g(\nab XX,Y)&=&g([Y,X],X)\\
&=&\trace (YX-XY)X^t\\
&=&\trace Y(XX^t-X^tX)^t\\
&=&g([X,X^t],Y).
\end{eqnarray*}
Let $[X,X^t]_\g$ be the orthogonal projection of the bracket
$[X,X^t]$ onto $\g$ in $\glr n$.  Then the above calculation shows
that $$\nab XX=[X,X^t]_\g$$ so the Laplacian $\tau$ satisfies
$$\tau(\phi)=\sum_{X\in\B}X^2(\phi)-[X,X^t]_\g(\phi).$$

\section{The Nilpotent Lie Group $N_n$}

The standard example of a nilpotent Lie group is the subgroup
$N_n$ of $\GLR n$ consisting of $n\times n$ upper-triangular
unipotent matrices
$$N_n=\Bigg\{\begin{pmatrix}
        1 & x_{12} & \cdots & x_{1,n-1}& x_{1n}
\\      0 &      1 & \ddots &          & \vdots
\\ \vdots & \ddots & \ddots &   \ddots & \vdots
\\ \vdots &        & \ddots &        1 & x_{n-1,n}
\\      0 & \cdots & \cdots &        0 & 1\end{pmatrix}
\in\GLR n \ |\ x_{ij}\in\rn\Bigg\}.$$ This inherits a natural
left-invariant Riemannian metric from $\GLR n$. The Lie algebra
$\n_n$ of $N_n$ has the canonical orthonormal basis $\B=\{E_{rs}|\
r<s\}$ and its Levi-Civita connection $\nabla$ satisfies
$$\nab{E_{rs}}{E_{rs}}=[E_{rs},E_{rs}^t]_{\n_n}=(E_{rr}-E_{ss})_{\n_n}=0$$
for all $E_{rs}\in\B$.  Hence the Laplacian $\tau$ is given by
$$\tau(\phi)=\sum_{r<s}E_{rs}^2(\phi).$$

\begin{lemma}\label{lemm:nilpotent0}
Let $x_{ij}:N_n\to\rn$ be the real valued coordinate functions
$$x_{ij}:x\mapsto e_i\cdot x\cdot e_j^t$$ where $\{e_1,\dots
,e_n\}$ is the canonical basis for $\rn^n$. If $i<j$ then the
following relations hold
$$\tau(x_{ij})=0\ \ \text{and}\ \
\kappa(x_{ij},x_{kl})=\delta_{jl}\cdot\hskip
-.5cm\sum_{\max\{i,k\}\le r<l}\hskip -.5cm x_{ir}x_{kr}.$$
\end{lemma}

\begin{proof}
For an element $X$ of the Lie algebra $\n_n$ we have
$$X(x_{ij}):x\mapsto e_i\cdot x\cdot X\cdot e_j^t\ \ \text{and}
\ \  X^2(x_{ij}):x\mapsto e_i\cdot x\cdot X^2\cdot e_j^t.$$ This
leads to the following
$$\tau(x_{ij})=\sum_{r<s}E_{rs}^2(x_{ij})=
\sum_{r<s}e_i\cdot x\cdot E_{rs}^2\cdot e_j^t=0,$$
\begin{eqnarray*}
\kappa(x_{ij},x_{kl})&=&\sum_{r<s}E_{rs}(x_{ij})E_{rs}(x_{kl})\\
&=&\sum_{r<s}e_i\cdot x\cdot E_{rs}\cdot e_j^t
\cdot e_l\cdot E_{rs}^t\cdot x^t\cdot e_k^t\\
&=&\sum_{r<s}e_i\cdot x\cdot E_{rs}\cdot E_{jl}\cdot
E_{rs}^t\cdot x^t\cdot e_k^t\\
&=&\sum_{r<s}\delta_{sj}\delta_{sl}\ e_i\cdot x\cdot E_{rr}\cdot
x^t\cdot
e_k^t\\
&=&\delta_{jl}\sum_{r<l}e_i\cdot x\cdot e_r^t\cdot e_r\cdot x^t\cdot e_k^t\\
&=&\delta_{jl}\cdot\hskip -.5cm\sum_{\max\{i,k\}\le r<l}\hskip
-.5cmx_{ir}x_{kr}.
\end{eqnarray*}
\end{proof}

\begin{theorem}\label{theo:nilpotent0}
Let $N_n$ be the nilpotent Lie group of real $n\times n$
upper-triangular unipotent matrices equipped with its standard
Riemannian metric.  Then the group epimorphism $\Phi:N_n\to\rn^{n-1}$
defined by $$\Phi(x)=(x_{12},\dots,x_{n-1,n})$$
is a harmonic morphisms.
\end{theorem}

\begin{proof}  As a direct consequence of Lemma \ref{lemm:nilpotent0}
we see that the components $\phi_1,\dots ,\phi_{n-1}$ of the
epimorphism $\Phi:N_n\to\rn^{n-1}$ satisfy the following system of
partial differential equations
$$\tau(\phi_k)=0\ \ \text{and}\ \ \kappa(\phi_k,\phi_l)=\delta_{kl}.$$
The result is a direct consequence of these formulae.
\end{proof}

\section{The First Construction}

In this section we generalize the above construction for $N_n$ to
a large class of Lie groups. We find an algebraic condition on the
Lie algebra which ensures the existence of at least local harmonic
morphisms.

\begin{proposition}\label{prop:general}
Let $G$ be a connected, simply connected Lie group with Lie
algebra $\g$ and non-trivial quotient algebra $\a=\g/[\g,\g]$ of
dimension $n$. Then there exists a natural group epimorphism
$\Phi:G\to\rn^n$ and left-invariant Riemannian metrics on $G$
turning $\Phi$ into a Riemannian submersion.
\end{proposition}

\begin{proof}
The two Lie algebras $\a$ and $\rn^n$ are Abelian so there exists
an isomorphism $\psi:\a\to\rn^n$ which lifts to a Lie group
isomorphism $\Psi:A\to\rn^n$ from the connected and simply
connected Lie group $A$ with Lie algebra $\a$.

On $\rn^n$ we have the standard Euclidean scalar product. Equip
the Lie algebra $\a$ with the unique scalar product turning $\psi$
into an isometry. This induces a left-invariant metric on $A$ and
the isomorphism $\Psi:A\to\rn^n$ is clearly an isometry.

Then equip the Lie algebra $\g$ of $G$ with {\it any} Euclidean
scalar product such that the projection map $\pi:\g\to\a$ is a
Riemannian submersion. This gives a Riemannian metric on $G$ and
the induced group epimorphism $\Pi:G\to A$ is a Riemannian
submersion. It follows from the above construction that the
composition $\Phi=\Psi\circ\Pi:G\to\rn^n$ is a group epimorphism
and its differential $d\Phi_e$ at $e$ is a Lie algebra
homomorphism. Furthermore $\Phi$ is a Riemannian submersion.
\end{proof}

\begin{theorem}\label{theo:general}
Let $G$ be a connected, simply connected Lie group with Lie
algebra $\g$ and non-trivial quotient algebra $\g/[\g,\g]$ of
dimension $n$. Let $g$ be a Riemannian metric on $G$ such that the
natural group epimorphism $\Phi:(G,g)\to\rn^n$ is a Riemannian
submersion. For the canonical basis $\{e_1,\dots,e_n\}$ of $\rn^n$
let $\{X_1,\dots,X_n\}$ be the orthonormal basis of the horizontal
space $\H_e$ with $d\Phi_e(X_i)=e_i$ and define the vector
$\xi\in\cn^n$ by $$\xi=(\trace\ad_{X_1},\dots,\trace\ad_{X_n}).$$
For a maximal isotropic subspace $W$ of $\cn^n$ put
$$V=\{w\in W\ |\ (w,\xi)=0\}.$$ If the dimension of the isotropic
subspace $V$ is at least $2$ then
$$\Omega_V=\{\phi_v(x)=(\Phi(x),v)\ |\ v\in V\}$$ is an orthogonal
family of globally defined harmonic morphisms on $(G,g)$.
\end{theorem}

\begin{proof}  First of all we note that for $X\in\g$ we have
\begin{eqnarray*}
X^k(\Phi)(p)&=&\frac {d^k}{dt^k}(\Phi(L_p(\exp(tX)))|_{t=0}\\
&=&\frac {d^k}{dt^k}(\Phi(p)+\Phi(\exp(tX)))|_{t=0}\\
&=&\frac {d^k}{dt^k}(t\Phi(X))|_{t=0}. \end{eqnarray*}  Then fix
an orthonormal basis $\B=\B_1\cup\B_2$ where $\B_1$ is an
orthonormal basis for $[\g,\g]$ and $\B_2$ for the orthogonal
complement $[\g,\g]^\perp$.  The tension field of $\Phi$ can now
be calculated as follows:
\begin{eqnarray*}
\tau(\Phi)(p)&=&\sum_{X\in\B}X^2(\Phi)(p)-d\Phi_p(\nab XX)\\
&=& -d\Phi_e(\sum_{X\in\B}\nab{X}{X})\\
&=&-d\Phi_e(\sum_{X,Z\in\B}\ip{\nab{X}{X}}{Z}Z)\\
&=&\sum_{Z\in\B_2}\sum_{X\in\B}\ip{[Z,X]}{X}d\Phi_e(Z)\\
&=&\sum_{Z\in\B_2}(\trace\ad_Z)d\Phi_e(Z).
\end{eqnarray*}

\end{proof}

\section{Nilpotent Lie Groups}

In this section we show that every connected, simply connected
nilpotent Lie group $G$ with Lie algebra $\g$ can be equipped with
natural Riemannian metrics admitting complex valued harmonic
morphisms on $G$. When the algebra is Abelian the problem is
completely trivial, so we shall assume that $\g$ is not Abelian.
In that case we have the following well-known fact.

\begin{lemma}\label{lemm:nilpotent}
Let $\g$ be a non-Abelian nilpotent Lie algebra.  Then the
dimension of the quotient algebra $\g/[\g,\g]$ is at least $2$.
\end{lemma}

\begin{proof}
Since the algebra $\g$ is nilpotent we have $[\g,\g]\neq\g$.
Assume that the quotient algebra $\g/[\g,\g]$ is of dimension $1$
i.e.
$$\g=\rn X\oplus[\g,\g]$$ for some $X\in\g$. Then
$[\g,\g]\subseteq [\g,[\g,\g]]$ and since of course
$[\g,[\g,\g]]\subseteq[\g,\g]$ we must have
$[\g,\g]=[\g,[\g,\g]]$. But as $\g$ is nilpotent, this is only
possibible if $[\g,\g]=0$ i.e. if $\g$ is Abelian.
\end{proof}

For non-Abelian nilpotent Lie groups the existence result of
Theorem \ref{theo:general} simplifies to the following.

\begin{theorem}\label{theo:nilpotent}
Let $G$ be a connected, simply connected, non-Abelian and
nilpotent Lie group with Lie algebra $\g$. Then there exist
Riemannian metrics $g$ on G such that the natural group
epimorphism $\Phi:(G,g)\to\rn^n$ is a Riemannian submersion. If
$V$ is a maximal isotropic subspace of $\cn^n$ then
$$\Omega_V=\{\phi_v(x)=(\Phi(x),v)\ |\ v\in V\}$$ is an orthogonal
family of globally defined harmonic morphisms on $(G,g)$.
\end{theorem}

\begin{proof}
The Lie algebra $\g$ is nilpotent, so if $Z\in\g$ then
$\trace\ad_Z=0$.  This means that the vector $\xi$ defined in
Theorem \ref{theo:general} vanishes.  The result is then a direct
consequence of Lemma \ref{lemm:nilpotent}.
\end{proof}

\section{The Nilpotent Heisenberg Group $H_n$}

We shall now apply Theorem \ref{theo:nilpotent} to yield an
orthogonal family of complex valued harmonic morphisms from the
well-known $(2n+1)$-dimensional nilpotent Heisenberg group
$$H_n=\Bigg\{\begin{pmatrix}
1 &   x & z\\
0 & I_n & y\\
0 &   0 & 1
\end{pmatrix}\in N_{n+2}|\ x,y^t\in\rn^n,\ z\in\rn\Bigg\},$$ where $I_n$
is the $n\times n$ identity matrix. A canonical orthonormal basis
$\B$ for the nilpotent Lie algebra $\h_n$ of $H_n$ consists of the
following matrices
$$X_k=\begin{pmatrix}
0 & e_k & 0\\
0 & 0_n & 0\\
0 &   0 & 0
\end{pmatrix},\
Y_k=\begin{pmatrix}
0 &   0 & 0\\
0 & 0_n & e_k^t\\
0 &   0 & 0
\end{pmatrix},\
Z=\begin{pmatrix}
0 &   0 & 1\\
0 & 0_n & 0\\
0 &   0 & 0
\end{pmatrix},$$
where $\{e_1,\dots ,e_n\}$ is the canonical basis for $\rn^n$. The
derived algebra $[\h_n,\h_n]$ is generated by $Z$ and the quotient
$\h_n/[\h_n,\h_n]$ can be identified with the $2n$-dimensional
subspace generated by $X_1,\dots,X_n,Y_1,\dots,Y_n$. The natural
group epimorphism $\Phi:H_n\to\rn^{2n}$ is given by
$$\Phi:\begin{pmatrix}
1 &   x & z\\
0 & I_n & y\\
0 &   0 & 1
\end{pmatrix}\to(x_1,\dots,x_n,y_1,\dots,y_n).$$

\begin{theorem}\label{theo:real-Heisenberg}
Let $H_n$ be the $(2n+1)$-dimensional Heisenberg group and
$\Phi:H_n\to\rn^{2n}$ be the natural group epimorphism. If $V$ is
a maximal isotropic subspace of $\cn^{2n}$ then
$$\Omega_V=\{\phi_v:x\mapsto (\Phi(x),v)\ |\ v\in V\}$$is an
orthogonal family of globally defined harmonic morphisms on $H_n$.
\end{theorem}

In \cite{Bai-Woo-1} Baird and Wood study the existence of harmonic
morphisms from the $3$-dimensional Heisenberg group $Nil=H_1$ to
surfaces. They show that up to conformal transfomations on the
codomain the only solutions defined locally on $H_1$ are the
restrictions to the natural group epimorphism $\phi:H_1\to\cn$
with
$$\phi:\begin{pmatrix}
        1 & x & z
\\      0 & 1 & y
\\      0 & 0 & 1
\end{pmatrix}\mapsto x+iy.$$

\section{The Nilpotent Lie Group $K_n$}

In this section we employ Theorem \ref{theo:nilpotent} to
construct a harmonic morphism on the $(n+1)$-dimensional nilpotent
subgroup $K_n$ of $\GLR {n+1}$ given by
$$K_n=\Bigg\{\begin{pmatrix}
1 &      x & p_2(x) & p_3(x) & \cdots & p_{n-1}(x) &     y_1\\
0 &      1 &      x & p_2(x) & \cdots & p_{n-2}(x) & y_{2}\\
  & \ddots & \ddots & \ddots & \ddots &     \vdots &  \vdots\\
  &        &      0 &      1 &      x &     p_2(x) &     y_{n-2}\\
  &        &        &      0 &      1 &          x &     y_{n-1}\\
  &        &        &        &      0 &          1 &     y_n\\
  &        &        &        &        &          0 &       1\\
\end{pmatrix}\ |\
x,y_{k}\in\rn\Bigg\},$$ where the polynomials $p_2,\dots,p_{n-1}$
are given by $p_k(x)=x^k/k!$ For the Lie algebra
$$\k_n=\Bigg\{\begin{pmatrix}
0 & \alpha &      0 &        &        &      0 &     \beta_1\\
  &      0 & \alpha &      0 &        &      0 & \beta_{2}\\
  & \ddots & \ddots & \ddots & \ddots & \vdots &      \vdots\\
  &        &        & 0      & \alpha &      0 &     \beta_{n-2}\\
  &        &        &        &    0   & \alpha &     \beta_{n-1}\\
  &        &        &        &        &      0 &     \beta_{n}\\
  &        &        &        &        &        &       0\\
\end{pmatrix} \in\glr {n+1} \ |\
\alpha,\beta_{k}\in\rn\Bigg\}$$ of $K_n$ we have the orthonormal
basis $\B$ consisting of the matrices
$$Y_1=E_{1,n+1},\dots,Y_n=E_{n,n+1}\ \ \text{and}\ \  X=\frac
1{\sqrt{n-1}}(E_{12}+\cdots + E_{n-1,n}).$$ The derived algebra
$[\k_n,\k_n]$ is generated by the vectors $Y_1,\dots,Y_{n-1}$ and
the quotient algebra $\k_n/[\k_n,\k_n]$ can be identified with the
$2$-dimensional subspace of $\k_n$ generated by $X$ and $Y_{n}$.
The natural group epimorphism $\Phi:K_n\to\cn$ given by
$$\Phi:p\mapsto (x\sqrt{n-1}+iy_n)$$ is a globally defined
harmonic morphism on $K_n$.

\section{Compact Nilmanifolds}

A {\it nilmanifold} is a homogeneous space of a nilpotent Lie group. As
is well known, any compact  nilmanifold is of the form $G/\Gamma$, where
$G$ is a nilpotent Lie group and $\Gamma$ a {\it uniform} subgroup of $G$,
i.e., a co-compact discrete subgroup. We have proved the global existence of
harmonic morphisms on any nilpotent Lie group with a left-invariant metric.
We include here a section where we prove the existence of a globally defined
harmonic morphism on any compact nilmanifold $G/\Gamma$, for which $G$ has
rational structure constants.

To begin with, assume that $(M,g)$ is a compact Riemannian manifold and
$\omega_1,\dots,\omega_k$ be a basis for the linear space of
harmonic 1-forms on $M$; thus $k$ is the first Betti number of $M$.
Define the lattice $\Lambda$ in $\rn^k$ as the integer span of the vectors
$$(\int_\gamma\omega_1,\dots,\int_\gamma\omega_k)\ \ \text{with}\ \ \gamma
\in H_1(M,\zn).$$ After fixing a point $p\in M$ we can define the
correponding  {\it Albanese map}
$$\pi_p:M\to\rn^k/\Lambda,\ \pi_p(q)=(\int_p^q\omega_1,\dots,
\int_p^q\omega_k)\ \text{mod}\ \Lambda.$$ Note that $\pi_p$ is a
harmonic map with respect to the flat metric on the torus
$\rn^k/\Lambda$.

To continue our construction, we need the following simple lemma,
where $\L$ denotes Lie derivative.
\begin{lemma}\label{lemm::invariant} Assume that $\omega$ is a
harmonic $p$-form on a compact Riemannian manifold $(M,g)$, and $X$ a
Killing vector field. Then $$\L_X\omega=0.$$
\end{lemma}
\begin{proof} Let $\varphi_t$ be the flow of $X$. As $\varphi_t$ is
an isometry, $\varphi_t^*\omega$ is also harmonic. Hence
$L_X\omega=d\iota_X\omega$ is a harmonic $p$-form, but as
it is also exact, it must vanish.
\end{proof}

Let $G$ be a Lie group and $\Gamma$ be a co-compact, discrete
subgroup of $G$. We equip $G$ with a left-invariant metric and
$G/\Gamma$ with the metric which makes the quotient map $G\to
G/\Gamma$ into a Riemannian submersion. Thus $G$ acts by
isometries on $G/\Gamma$. By the above lemma, any harmonic
$1$-form on $G/\Gamma$ is left-invariant by $G$, and so has
constant pointwise norm. Hence, any two harmonic $1$-forms
which are orthogonal at one point, remain orthogonal everywhere.
Thus we can choose $\omega_1,\dots,\omega_k$ in the above
construction such that the relation
$$\ip{\omega_i}{\omega_j}=\delta_{ij}$$ holds everywhere on
$G/\Gamma$. This makes the Albanese map into a Riemannian harmonic
submersion and hence a harmonic morphism.

Malcev proves in \cite{Mal} that a simply connected nilpotent Lie
group $G$ contains a uniform subgroup $\Gamma$ if and only if the
Lie algebra $\g$ of $G$ has rational structure constants with
respect to some basis.  It is known that this condition is always
satisfied if the dimension of $\g$ is less than $7$. In \cite{Nom}
Nomizu shows that the first Betti number of the quotient
$G/\Gamma$ equals the codimension of $[\g,\g]$ in $\g$, which by
Lemma \ref{lemm:nilpotent} is at least $2$ unless the algebra is
Abelian. The above arguments now deliver the following result.

\begin{theorem} Let $G$ be a non-Abelian nilpotent Lie group
for which the Lie algebra $\g$ has rational structure constants in
some basis. For any uniform subgroup $\Gamma$ and any $G$-invariant
Riemannian metric $g$ on $G/\Gamma$, there exists a harmonic morphism
from $(G/\Gamma,g)$ into a flat torus of dimension at least $2$.
\end{theorem}

To show that the lift of the Albanese map to a map between the
universal coverings is essentially the same as we constructed in
the previous section, we show the following result.

\begin{proposition}
Assume that $G$ is a connected and simply connected Lie group, and
$\Gamma$ a uniform subgroup of $G$, and equip $G$ and $G/\Gamma$
with $G$-invariant metrics for which the covering map $G\to
G/\Gamma$ is a local isometry. Then any harmonic map
$\phi:G\to\rn^n$ which is right-invariant under the action of
$\Gamma$ is a homomorphism followed by a translation in $\rn^n$.
\end{proposition}

\begin{proof}
By translating, we may assume that $\phi(e)=0$. Consider the
$1$-form $\omega=d\phi$ on $G$. This is the lift of a harmonic
$1$-form on $G/\Gamma$, and so, by Lemma \ref{lemm::invariant}, is
left-invariant on $G$. Fix a $g\in G$ and define
$\Phi(h)=\phi(g)+\phi(h)$ and $\Psi(h)=\phi(gh)$. As $d\phi$ is
left-invariant, it follows easily that $d\Phi=d\Psi$, and since
$\Phi(e)=\Psi(e)$, we get $\Phi=\Psi$. Hence $\phi$ is a
homomorphism.
\end{proof}

Thus, the lift of the Albanese map is a homomorphism and a
Riemannian submersion $\phi:G\to\rn^n$, where $n=\dim\g/[\g,\g]$.
Since the kernel of $d\phi$ necessarily contains $[\g,\g]$, these
spaces must coincide. This shows that $\phi$, up to a translation
in $\rn^n$, is precisely the map constructed in Theorem
\ref{theo:nilpotent}.

\section{The Solvable Lie Group $S_n$}

The standard example of a solvable Lie group is the subgroup $S_n$
of $\GLR n$ of $n\times n$ upper-triangular matrices.  This
inherits a natural left-invariant Riemannian metric from $\GLR n$.
The connected component $G$ of $S_n$ containing the neutral
element $e$ is given by
$$G=\Bigg\{\begin{pmatrix}
e^{t_1} & x_{12}  & \cdots &   x_{1,n-1} & x_{1n}\\
    0   & e^{t_2} & \cdots &   x_{2,n-1} & x_{2n}\\
 \vdots & \ddots  & \ddots &     \vdots  & \vdots\\
    0   & \cdots  &   0    & e^{t_{n-1}} & x_{n-1,n}\\
    0   & \cdots  &   0    &          0  & e^{t_n}
\end{pmatrix}\in\GLR n \ |\
x_{ij},t_i\in\rn\Bigg\}.$$

The Lie algebra $\g$ of $G$ consist of all upper-triangular
matrices and has the orthogonal decomposition $\g =\d\oplus\n_n$
where $\n_n$ is the Lie algebra for $N_n$ and $\d$ is generated by
the diagonal elements $D_1=E_{11},\dots,D_n=E_{nn}\in\g$.  The
derived algebra $[\g,\g]$ is $\n_n$ and the quotient $\g/[\g,\g]$
can be identified with $\d$. The natural group epimorphism
$\Phi:G\to\rn^n$ is the Riemannian submersion given by
$$\Phi:g\to(t_1,\dots,t_n).$$

\begin{theorem}\label{theo:upper}
Let $G$ be the connected component of the solvable Lie group $S_n$
containing the neutral element $e$ and $\Phi:G\to\rn^{n}$ be the
natural group epimorphism.  Let the vector $\xi\in\cn^n$ be given
by $$\xi=((n+1)-2,(n+1)-4,\dots,(n+1)-2n).$$ If $n\ge 3$, $W$ is a
maximal isotropic subspace of $\cn^n$ and $$V=\{w\in W\ |\
(w,\xi)=0\},$$ then $$\Omega_V=\{\phi_v(x)=(\Phi(x),v)\ |\ v\in
V\}$$ is an orthogonal family of globally defined harmonic
morphisms on $G$.
\end{theorem}

\begin{proof}
The subspace $\d$ of $\g$ is the horizontal space $\H_e$ of
$\Phi:G\to\rn^n$ and has the orthonormal basis
$\{D_1,\dots,D_n\}$. For the diagonal elements $D_t$ we have
$$\trace\ad_{D_t}=(n+1)-2t$$ which gives
\begin{eqnarray*}
\xi&=&(\trace\ad_{X_1},\dots,\trace\ad_{X_n})\\
&=&((n+1)-2,(n+1)-4,\dots,(n+1)-2n).
\end{eqnarray*}
This means that for any maximal isotropic subspace $W$ of $\cn^n$
the space $V=\{w\in W\ |\ (w,\xi)=0\}$ is at least two
dimensional.  The result is then a consequence of Theorem
\ref{theo:general}.
\end{proof}

\section{Symmetric Spaces of Rank $r\ge 3$}

In this section we employ Theorem \ref{theo:general} to construct
complex valued harmonic morphisms from Riemannian symmetric spaces
of rank at least 3.  For the details of their structure theory we
refer to \cite{Hel}. In the next section we will construct harmonic
morphisms on symmetric spaces of rank one through a different method.

An irreducible Riemannian symmetric space of
non-compact type may be written as $G/K$ where $G$ is a simple,
connected and simply connected Lie group and $K$ is a maximal compact
subgroup of $G$. We denote by $\g$ and $\k$ the Lie algebras of
$G$ and $K$, respectively, and by $$\g=\k+\p$$ the corresponding Cartan
decomposition of $\g$. According to the Iwasawa decomposition
of $\g$ we have $$\g=\n+\a+\k,$$ where $\a$ is a maximal Abelian
subalgebra of $\p$ and $\n$ a nilpotent subalgebra of $\g$.
Furthermore, the subalgebra $$\s=\n+\a$$ is a solvable subalgebra
of $\g$. On the group level we have similar decompositions
$$G=NAK\ \ \text{and}\ \ S=NA$$ where $N$ is a
normal subgroup of $S$.

As $S$ acts simply transitively on $G/K$ we thus obtain a
diffeomorphism $$G/K\cong S$$ and, as the action of $S$ on
$G/K$ is isometric, the induced metric on $S$ is left-invariant.

Since $[\s,\s]=\n$, we see that the codimension of
the derived algebra of $\s$ is the dimension of $\a$ i.e. the
rank of $G/K$. Hence, when the rank is at least 3, the statement
of Theorem   \ref{theo:general} shows that there exist globally
defined complex valued harmonic morphisms on $G/K$.

\begin{theorem}\label{theo:existence}
Let $(M,g)$ be an irreducible Riemannian symmetric space of
rank at least $3$. Then for each point $p\in M$ there exists a
complex-valued harmonic morphism $\phi:U\to\cn$ defined on an open
neighbourhood $U$ of $p$. If the space $(M,g)$ is of non-compact
type then the domain $U$ can be chosen to be the whole of $M$.
\end{theorem}

\begin{proof}
As for the compact situtation, the duality principle decscribed
in \cite{Gud-Sve-1} gives us locally defined complex valued harmonic
morphisms on the irreducible Riemannian symmetric spaces
of compact type with rank at least $3$.
\end{proof}

Theorem \ref{theo:existence} gives the first known examples of
complex-valued harmonic morphisms from the non-compact classical
Riemannian symmetric spaces
\begin{gather*}
\SOO{p}{p}/\SO p\times\SO p,\\
\SOO{p}{p+1}/\SO p\times\SO{p+1},\\
\Spp{p}{p}/\Sp p\times\Sp p
\end{gather*}
with $p\geq 3$ and their compact duals. As for the exceptional
symmetric spaces, the result gives the first known examples from
the non-compact
\begin{gather*}
E_6^2/\SU 6\times\SU 2,\\
E_7^{-1}/\SO{12}\times\SU 2,\\
E_8^{-24}/E_7\times\SU 2,\\
F_4^4/\Sp 3\times\SU 2,\\
E_6^6/\Sp 4,\\
E_8^8/\SO{16}.
\end{gather*}
and their compact duals.

\section{The Second Construction}

In this section we introduce a new method for constructing complex valued 
harmonic morphisms from a certain class of solvable Lie groups.  This provides 
solutions on any Damek-Ricci space and many Riemannian symmetric spaces.

Let $\s$ be a Lie algebra which can be decomposed as the
sum of two subalgebras $$\s=\n+\a,\quad \n\cap\a=\{0\},$$ where $\n$ is a
nilpotent ideal and $\a$ is Abelian. Furthermore let $\s$ be
equipped with an inner product for
which $\n$ and $\a$ are orthogonal and the adjoint action of $\a$ on
$\n$ is self-adjoint with respect to this inner product. Thus, we have a
subset of ``roots'' $\Sigma$ in the dual space $\a^*$ of $\a$ and an
orthogonal
decomposition $$\n=\sum_{\alpha\in\Sigma}\n_\alpha,\ \text{  where }\
[V,X]=\alpha(V)X$$ for any $V\in\a$ and $X\in\n_\alpha$. Finally, we
assume that there is at least one root $\beta\in\Sigma$ with the
property that $$\n_\beta\perp [\n,\n].$$ We write
$\Sigma^*=\Sigma\setminus\{\beta\}$ and put
$$\m=\sum_{\alpha\in\Sigma^*}\n_\alpha.$$

Let $S$ be a simply connected Lie group with Lie algebra $\s$ and
$A$, $N$ and $M$ be the Lie subgroups of $S$ with Lie algebras $\a$,
$\n$ and $\m$, respectively. Then $S=N\cdot A$ and $M$ is a closed normal
subgroup of $N$ with $N/M$ Abelian and $A\cdot M=M\cdot A$. Furthermore,
$A$, $M$ and $N$ are
simply connected, hence so is the quotient $N/M$.

We equip $S$ with the left-invariant Riemannian metric induced by
the inner product on $\s$ and give $N/M$ the unique left-invariant metric for
which the homogeneous projection $N\to N/M$ is a Riemannian
submersion.

\begin{theorem} For the above situation the map $$\Psi:S=N\cdot A\to N/M,\quad
\Psi(n\cdot
  a)=n\cdot M$$ is a harmonic morphism.
\end{theorem}

\begin{proof} Denote by $o=eM$ the identity coset of $N/M$.
The fibre of $\Psi$ over a point $n\cdot o$ is the
set $n\cdot M\cdot A=n\cdot A\cdot M$. Thus, the vertical space of
$\Psi$ at the point $n\cdot a$ is the left translate of $\a+\m$,
i.e.  $$\V_{n\cdot a}=d(L_{n\cdot a})_e(\a+\m);$$ the horizontal
space is therefore the left translate of $\n_\beta$:
$$\H_{n\cdot a}=d(L_{n\cdot a})_e(\n_\beta).$$ This means that for
$X\in\n_\beta$ we have
\begin{equation*}
\begin{split}
d\Psi_{n\cdot a}(d(L_{n\cdot a})_e(X))=&\frac{d}{dt}\big\vert_{t=0}
\Psi(n\cdot a\cdot\exp(tX))\\
=&\frac{d}{dt}\big\vert_{t=0}n\cdot\exp(t\Ad_a(X))\cdot o\\
=&d(L_n)_o(\Ad_a(X))
\end{split}
\end{equation*}
Now, there is a unique element $V\in\a$ such that $a=\exp(V)$. Thus,
$$\|d\Psi_{n\cdot a}(d(L_{n\cdot
  a})_e(X))\|^2=\|\Ad_a(X)\|^2=\|e^{\ad(V)}(X)\|^2=e^{2\beta(V)}\|X\|^2.$$ This
implies that $\Psi$
is horizontally conformal with dilation $$\lambda^2(n\cdot
a)=e^{2\beta(V)}$$ so in particular, $\Psi$ is horizontally
homothetic.  This means that for proving that $\Psi$
is a harmonic morphism it is sufficient to show that it has minimal
fibres.  In our special situation it is even enough to show that the
fibre over $o\in N/M$ is minimal.  Since
$\Psi^{-1}(o)=M\cdot A$ is a subgroup of $S$, it even sufficies to check
that it is minimal at the identity element $e\in S$.

Let $\{f_i\}$ be an orthonormal basis of $\a$ and for any
$\alpha\in\Sigma^*$ let
$\{e_k^\alpha\}$ be an orthonormal basis for $\n_\alpha$. Then
$\{f_i,e_k^\alpha\}$ is an orthonormal basis for the sum $\a+\m$.
Let $H$ be the mean curvature vector field along the fibre $\Psi^{-1}(o)$ and
$X\in\n_\beta$ as above, then
\begin{equation*}
\begin{split}
\ip HX=&\sum_{i}\ip{\nabla_{f_i}f_i}{X}+\sum_{\alpha\in\Sigma^*}
\sum_k\ip{\nabla_{e_k^\alpha}e_k^\alpha}{X}\\
=&\sum_i\ip{[X,f_i]}{f_i}
+\sum_{\alpha\in\Sigma^*}\sum_k\ip{[X,e_k^\alpha]}{e_k^\alpha}\\
=&0,
\end{split}
\end{equation*}
since $[X,f_i]\in\n_\beta\perp\a$ and
$[X,e_k^\alpha]\in\n_{\beta+\alpha}\perp\n_\alpha$.
\end{proof}

Note that, since $N/M$ is Abelian and simply connected, there is an isometric group
isomorphism $$N/M\cong\rn^k,$$ where $k=\dim\n_\beta$.

\begin{example} Consider an irreducible Riemannian symmetric space $G/K$ of
non-compact type. We recall again the Iwasawa decomposition $$G=N\cdot A\cdot
K,\quad \g=\n+\a+\k$$ of the group $G$ and its Lie algebra $\g$. Choose a root space
$\n_\beta\subset\n$ associated to a simple, restricted root $\beta$. The above
construction now yields a harmonic morphism $$G/K\cong N\cdot A\to N/M\cong\rn^k,$$
where $k=\dim\n_\beta$. Thus, whenever $\beta$ can be chosen with $k\geq2$, there
are globally defined, complex-valued harmonic morphisms on $G/K$, regardless of the
rank. Besides giving new examples on some of the spaces considered in the previous
section, we also find the first known complex valued harmonic morphisms on the
spaces $$\Spp{2}{2}/\Sp{2}\times\Sp{2},\quad E_6^{-26}/F_4,\quad
F_4^{-20}/\Spin{9},$$
and their compact duals.
\end{example}

\begin{example} Recall that a \emph{Damek-Ricci space} is a simply connected
solvable Lie group $S$ with a left-invariant metric, such that its Lie algebra is an
orthogonal direct sum $$\s=\v+\z+\a,$$ where $\n=\v+\z$ is a  generalized Heisenberg
algebra with center $\z$, $\a$ is 1-dimensional with a fixed, non-zero vector
$A\in\a$, and
$$[A,V]=\frac{1}{2}V,\quad [A,Y]=Y\qquad(A\in\a,\ V\in\v,\ Y\in\z).$$  Clearly,
these spaces generalize non-compact Riemannian symmetric spaces of rank one. They
are irreducible Riemannian harmonic manifolds, and were introduced by Damek and
Ricci to provide counterexamples to the conjecture by Lichnerowicz, stating that any
Riemannian harmonic manifold is 2-point symmetric \cite{DaRi}. For more information
about Damek-Ricci spaces and generalized Heisenberg algebras, we refer to
\cite{BeTrVa}.

Let $A$, $N$, $Z$ be the subgroups of $S$ with Lie algebras $\a$, $\n$ and $\z$,
respectively. The above construction gives us a harmonic morphism $$\varphi:S=N\cdot
A\to N/Z\cong\rn^k,$$ where $k=\dim\v$.
\end{example}

\section{3-dimensional Solvable Lie Groups}

In this section we study the structure of conformal foliations by
geodesics on $3$-dimensional solvable Lie groups. We start by
reviewing some general terminology.

Assume that $\V$ is an involutive distribution on a
Riemannian manifold $(M,g)$ and denote by $\H$ its orthogonal
complement. As customary, we also use $\V$ and $\H$ to denote the
orthogonal projections onto the corresponding subbundles of $TM$
and we identify $\V$ with the corresponding foliation tangent to
$\V$. The second fundamental form for $\V$ is given by
$$B^\V(U,V)=\frac 12\H(\nabla_UV+\nabla_VU)\qquad(U,V\in\V),$$
while the second fundamental form for $\H$ is given by
$$B^\H(X,Y)=\frac{1}{2}\V(\nabla_XY+\nabla_YX)\qquad(X,Y\in\H).$$
Recall that $\V$ is said to be \emph{conformal} if there is a
vector field $V$, tangent to $\V$, such that $$B^\H=g\otimes V,$$ and
$\V$ is said to be \emph{Riemannian} if $V=0$.
Furthermore, $\V$ is said to be \emph{totally geodesic} if
$B^\V=0$. This is equivalent to the leaves of $\V$ being totally
geodesic submanifolds of $M$.

It is easy to see that the fibres of a horizontally conformal
map (Riemannian submersion) give rise to a conformal (Riemannian)
foliation. Conversely, any conformal (Riemannian) foliation is
locally the fibres of a horizontally conformal map (Riemannian submersion),
see \cite{Bai-Woo-book}. When the codimension of the foliation is
$2$, the map is harmonic if and only if the leaves of the
foliation are minimal submanifolds.

Before restricting ourselves to the 3-dimensional situation,
we note the following result.

\begin{proposition}\label{prop:center}
Let $G$ be a Lie group with Lie algebra $\g$ and a
left-invariant metric. Then the left-translation of any
subspace of the centre of $\g$ generates a totally geodesic
Riemannian foliation of $G$.
\end{proposition}

\begin{proof}
Let $\V$ be the foliation thus obtained and $\H$ its
orthogonal complement. If $U$ and $V$ are left-invariant
vector field in $\V$ and $X$ and $Y$ left-invariant in $\H$
then clearly
$$\ip{B^\V(U,V)}{X}=\frac{1}{2}(\ip{[U,V]}{X}
+\ip{[X,U]}{V}-\ip{[V,X]}{U})=0$$
and
\begin{equation*}
\begin{split}
\ip{B^\H(X,Y)}{U}=&\frac{1}{2}(\ip{\nabla_XY}{U}+\ip{\nabla_YX}{U})\\
=&\frac{1}{4}(\ip{[X,Y]}{U}+\ip{[U,X]}{Y}-\ip{[Y,U]}{X}\\
&+\ip{[Y,X]}{U}+\ip{[U,Y]}{X}-\ip{[X,U]}{Y})=0.
\end{split}
\end{equation*}
Thus $\V$ is Riemannian and totally geodesic.
\end{proof}

\begin{example} The solvable Lie algebra $$\g=\Bigg\{
\begin{pmatrix} t_1 & x \\ 0 & t_2\end{pmatrix}
\Bigg|\ t_1,t_2,x\in\rn\Bigg\}$$
has a $1$-dimensional center. Thus the corresponding simply connected
Lie group $S_2$ admits a Riemannian foliation by geodesics, regardless of
which left-invariant metric we equip it with.
\end{example}

The following result shows that the 3-dimensional situation is
very special with respect to conformal foliations by geodesics.

\begin{theorem}\cite{Bai-Woo-book}\label{thm:uniqueness}
Let $M$ be a $3$-dimensional Riemannian manifold with
non-constant sectional curvature. Then there are at most two
distinct conformal foliations by geodesics of $M$. If there is
an open subset on which the Ricci tensor has precisely two distinct
eigenvalues, then there is at most one conformal foliation by
geodesics of $M$.
\end{theorem}

This result can now be applied to the situation of $3$-dimensional
Lie groups with left-invariant metrics.

\begin{proposition}\label{prop:foliation}
Let $G$ be a connected $3$-dimensional Lie group with a
left-invariant metric of non-constant sectional curvature. Then any
local conformal foliation by geodesics of a connected open
subset of $G$ can be extended to a global conformal foliation by
geodesics of $G$.  This is given by the left-translation of a
1-parameter subgroup of $G$.
\end{proposition}

\begin{proof}Assume that $\V$ is a conformal foliation by geodesics
of some connected neighbourhood $U$ of the identity element $e$ of
$G$ and denote by $\H$ the orthogonal complement of $\V$. Let
$U'\subset U$ be a connected neighbourhood of $e$ such that $gh\in
U$ for all $g,h\in U'$, and let $U''\subset U'$ be a connected
neigbourhood of $e$ for which $g^{-1}\in U'$ for all $g\in U''$.

For any $g\in G$, we denote by $L_g:G\to G$ left translation by
$G$. Take $g\in U''$ and consider the distribution
$dL_g\V\big\vert_{U'}$, obtained by restricting $\V$ to $U'$ and
translating with $g$. As $L_g$ is an isometry, this is also a
conformal foliation by geodesics of $L_g U'$, which is a connected
neigbourhood of $e$. It is clear from Theorem \ref{thm:uniqueness}
and by continuity, that this distribution must coincide with $\V$
restricted to $L_g U'$. It follows that $d(L_g)_h(\V_h)=\V_{gh}$
for all $g,h\in U''$. In particular we have
$$d(L_g)_e(\V_e)=\V_g\qquad(g\in U'').$$

Define a 1-dimensional distribution $\tilde\V$ on $G$ by
$$\tilde\V_g=(dL_g)_e(\V_e)\qquad(g\in G).$$ Its horizontal
distribution $\tilde\H$ is clearly given by left translation of
$\H_e$. From the above we see that
$$\tilde\V\big\vert_{U''}=\V\big\vert_{U''}.$$ It follows that
$$B^{\tilde\V}\big\vert_{U''}=B^{\V}\big\vert_{U''}=0,$$ and since
$\tilde\V$ is left-invariant, it follows that $B^{\tilde\V}=0$
everywhere, i.e., $\tilde\V$ is totally geodesic. In the same way
we see that $\tilde\V$ is a conformal distribution and, by Theorem
\ref{thm:uniqueness}, we see that $$\tilde\V\big\vert_{U}=\V.$$

This shows that $\V$ extends to a global conformal, totally
geodesic distribution $\tilde\V$, which is left-invariant. By
picking any unit vector $V\in \V_e$, we see that the corresponding
foliation is given by left translation of the 1-parameter subgroup
generated by $V$.
\end{proof}

\begin{example}\label{ex:conformal}
Fix two real numbers $\alpha,\beta$, not both zero, and let
$\g=\h\ltimes\rn^2$, where $\h\subset\glr{2}$ is the real span
of the matrix $$e_1=\begin{pmatrix} \alpha & -\beta \\ \beta &
  \alpha\end{pmatrix}.$$ Thus, if $e_2,e_3$ is
the standard basis for $\rn^2$, we have the commutator relations
$$[e_1,e_2]=\alpha e_2+\beta e_2,\ [e_1,e_3]=-\beta e_2+\alpha
e_3,\ [e_2,e_3]=0.$$
Hence $\g$ is solvable and centerless. The corresponding simply
connected Lie group is given by $$
G=\Bigg\{\begin{pmatrix} e^{t\alpha}\cos(t\beta) &
  -e^{t\alpha}\sin(t\beta) & x \\
e^{t\alpha}\sin(t\beta) & e^{t\alpha}\cos(t\beta) & y \\
0 & 0 & 1 \end{pmatrix}\ \Bigg|\ t,x,y\in\rn\Bigg\}.$$ Let $e_2,e_3$ be
the standard basis on $\rn^2$ and choose the left-invariant metric
for which $e_1,e_2,e_3$ is an orthonormal basis for $\g$. By
identifying $G$ with $\rn^3$ in
the obvious way, we see that this metric is given by $$dt^2
+e^{-2t\alpha}dx^2+e^{-2t\alpha}dy^2.$$
A simple calculation shows that this metric has constant
sectional curvature $-\alpha^2$. It is also
easy to see that left translation of $e_1$ generates a conformal
foliation by geodesics on $G$.
\end{example}

\begin{theorem}\label{theo:solv}
Let $\g$ be a $3$-dimensional centerless, solvable Lie algebra
and $G$ a connected Lie group with Lie algebra $\g$ and a
left-invariant metric. Let $\V$ be a local conformal foliation 
by geodesics on $G$. Then $G$ has constant sectional curvature.
\end{theorem}

\begin{proof}
According to Proposition \ref{prop:foliation} the foliation
$\V$ can be extended to a global foliation on $G$ and is
tangent to a left-invariant vector field $V\in\g$.
Let $\H$ be the left-invariant distribution orthogonal to $\V$
and $X,Y$ be a left-invariant orthonormal basis for $\H$.
As $\V$ is totally geodesic,
we yield $$0=\ip{B^\V(V,V)}{X}=\ip{[X,V]}{V}$$ and similarly for
$Y$. We thus get
\begin{equation*}
\begin{split}
[V,X]=&\alpha X+\beta Y\\
[V,Y]=&\gamma X+\delta Y.
\end{split}
\end{equation*}
As $\V$ is conformal we
have$$0=\ip{B^\H(X,X)}{V}-\ip{B^\H(Y,Y)}{V}=\ip{[V,X]}{X}
-\ip{[V,Y]}{Y}=\alpha-\delta$$
and
\begin{equation*}
\begin{split}
0=&\ip{B^\H(X,Y)}{V}=\frac{1}{2}\big\{\ip{[V,X]}{Y}
+\ip{[X,Y]}{V}-\ip{[Y,V]}{X}\\
&+\ip{[V,Y]}{X}+\ip{[Y,X]}{V}-\ip{[X,V]}{Y}\big\}\\
=&\ip{[V,X]}{Y}+\ip{[V,Y]}{X}=\beta+\gamma.
\end{split}
\end{equation*}
Thus
\begin{equation}\label{eq:adjoint}
\begin{split}
[V,X]=&\alpha X+\beta Y\\
[V,Y]=&-\beta X+\alpha Y.
\end{split}
\end{equation}
The Lie algebra $\g$ is centerless so we must have
$\alpha^2+\beta^2\neq 0$, and since
\begin{equation*}
\begin{split}
[V,\beta X+\alpha Y]=&(\alpha^2+\beta^2)Y\\
[V,\alpha X-\beta Y]=&(\alpha^2+\beta^2)X,
\end{split}
\end{equation*}
it follows that $X,Y\in[\g,\g]$. As $\g$ is solvable, we must have
$[\g,\g]\neq \g$, so $[\g,\g]=\mathrm{span}\{X,Y\}$. Since
$[\g,\g]$ is a $2$-dimensional, nilpotent Lie algebra, it must be
Abelian; hence $[X,Y]=0$.

By comparing with Example \ref{ex:conformal}, it now follows that
$G$ must have constant sectional curvature $-\alpha^2$.
\end{proof}

\begin{example}
Consider the $3$-dimensional solvable Lie algebra $\g_\alpha$
spanned by the matrices
$$e_1=\begin{pmatrix} \alpha & 0 & 0 \\ 0 & -1 & 0 \\ 0 & 0 & 0
\end{pmatrix},\ e_2=\begin{pmatrix} 0 & 0 & 1
\\ 0 & 0 & 0 \\ 0 & 0 & 0 \end{pmatrix},\ e_3=\begin{pmatrix} 0 &
0 & 0 \\ 0 & 0 & -1 \\ 0 & 0 & 0 \end{pmatrix},$$ where $\alpha$
is some fixed real number. Here $[\g,\g]$ is spanned by $e_2$ and
$e_3$, and the operator $\ad_{e_1}$ acts on this space with eigenvalues
$\alpha$ and $-1$. The simply connected Lie group $G_\alpha$
with Lie algebra $\g_\alpha$ is given by
$$G_\alpha=\Bigg\{\begin{pmatrix} e^{\alpha x} & 0 & y \\ 0 &
e^{-x} & z \\ 0 & 0 & 1 \end{pmatrix}|\ x,y,z\in\rn\Bigg\}.$$

When $\alpha=0$, $e_2$ spans the center of $\g_\alpha$, and is
thus tangent to a Riemannian foliation of $G_\alpha$ by geodesics.
For $\alpha\neq0$, $\g_\alpha$ is centerless so the only metrics
on $G_\alpha$ for $\alpha\neq0$ which admit a conformal
foliation by geodesics are those of constant sectional curvature.
For $\alpha>0$ there are no left-invaraint metric on $G_\alpha$
with constant sectional curvature, see \cite{Mil}. This implies that
for $\alpha>0$ the Lie group $G_\alpha$ does not admit any local conformal
foliation by geodesics, regardelss of which left-invariant metric
we equip it with.
\end{example}

\section{Acknowledgements}
The authors are grateful to Vincente Cort\'es and Andrew Swann
for useful discussions on this work.

\end{document}